\documentclass[oneside,reqno,american]{amsart}
\usepackage[T1]{fontenc}
\usepackage[utf8]{inputenc}
\usepackage{xcolor}
\usepackage{babel}
\usepackage{prettyref}
\usepackage{amstext}
\usepackage{amsthm}
\usepackage{amssymb}
\usepackage[pdfusetitle,
 bookmarks=true,bookmarksnumbered=false,bookmarksopen=false,
 breaklinks=false,pdfborder={0 0 0},pdfborderstyle={},backref=false,colorlinks=false]
 {hyperref}
\hypersetup{
 colorlinks=true,citecolor=blue,linkcolor=blue,linktocpage=true}

\makeatletter
\numberwithin{equation}{section}
\numberwithin{figure}{section}

\usepackage{prettyref}

\newrefformat{cor}{Corollary~\ref{#1}}
\newrefformat{subsec}{Section~\ref{#1}}
\newrefformat{lem}{Lemma~\ref{#1}}
\newrefformat{thm}{Theorem~\ref{#1}}
\newrefformat{sec}{Section~\ref{#1}}
\newrefformat{chap}{Chapter~\ref{#1}}
\newrefformat{prop}{Proposition~\ref{#1}}
\newrefformat{exa}{Example~\ref{#1}}
\newrefformat{tab}{Table~\ref{#1}}
\newrefformat{rem}{Remark~\ref{#1}}
\newrefformat{def}{Definition~\ref{#1}}
\newrefformat{fig}{Figure~\ref{#1}}
\newrefformat{claim}{Claim~\ref{#1}}

\makeatother

\theoremstyle{plain}
\newtheorem{thm}{\protect\theoremname}[section]
\theoremstyle{definition}
\newtheorem{defn}[thm]{\protect\definitionname}
\theoremstyle{remark}
\newtheorem{rem}[thm]{\protect\remarkname}
\theoremstyle{plain}
\newtheorem{lem}[thm]{\protect\lemmaname}
\newtheorem{prop}[thm]{\protect\propositionname}
\newtheorem*{question*}{\protect\questionname}
\newtheorem{cor}[thm]{\protect\corollaryname}
\theoremstyle{definition}
\newtheorem{example}[thm]{\protect\examplename}
\providecommand{\corollaryname}{Corollary}
\providecommand{\definitionname}{Definition}
\providecommand{\examplename}{Example}
\providecommand{\lemmaname}{Lemma}
\providecommand{\propositionname}{Proposition}
\providecommand{\questionname}{Question}
\providecommand{\remarkname}{Remark}
\providecommand{\theoremname}{Theorem}

\begin{document}
\title[Residual-Weighted Decomposition]{Residual-Weighted Decomposition of Positive Operators}
\begin{abstract}
This paper investigates an iterative rank-one decomposition scheme
for positive operators on a Hilbert space based on a residual-weighted
congruence update. At each step the operator is compressed along a
chosen unit vector while remaining inside the positive cone, and the
resulting map defines a monotone dynamical system on the cone of positive
operators. We prove that the associated residuals admit a canonical
telescoping decomposition into rank-one terms and a limiting positive
operator, and we identify this limit together with an exact energy
identity expressing the defect between the initial and limiting operators
as a convergent series of rank-one contributions. In the case where
the iteration exhausts the operator, the residual directions form
a Parseval frame for the natural range space, yielding a constructive
procedure that produces Parseval frames without spectral calculus.
We further solve the inverse problem by characterizing those decreasing
chains with rank-one steps that arise from such dynamics via an intrinsic
normalization condition involving the Moore-Penrose inverse. For trace-class
operators we obtain a scalar energy identity and show that mild greedy
or density conditions on the chosen directions guarantee exhaustion.
An application to reproducing kernel Hilbert spaces illustrates the
abstract results.
\end{abstract}

\author{James Tian}
\address{Mathematical Reviews, 535 W. William St, Suite 210, Ann Arbor, MI
48103, USA}
\email{jft@ams.org}

\maketitle
\tableofcontents{}

\section{Introduction}\label{sec:1}

Decomposing vectors in a Hilbert space is a well-understood problem,
with many constructive methods available. Classical Fourier series
and orthogonal expansions provide representations of the form
\[
x=\sum_{n\ge0}\left\langle u_{n},x\right\rangle u_{n},
\]
and modern greedy algorithms iteratively build similar expansions
with a residual that remains a vector. The update 
\[
r_{n+1}=r_{n}-\left\langle u_{n},r_{n}\right\rangle u_{n}
\]
does not leave the ambient space and naturally preserves the basic
geometry.

For positive operators, the situation is fundamentally different.
If we try to remove a rank-one term from a positive operator $R_{n}$
by an additive update 
\[
R_{n+1}=R_{n}-\alpha P
\]
($\alpha>0$, $P$ a projection), there is no guarantee that $R_{n+1}$
remains positive. Unless $\alpha$ is chosen using spectral information,
as in Cholesky-type factorization \cite{MR1098323,MR2738209,MR3024913,MR3509211},
the update may leave the cone $B\left(H\right)_{+}$. This basic obstruction
limits direct analogues of greedy algorithms in the operator setting;
the cone $B(H)_{+}$ is not closed under subtractions unless the removed
term is specifically dominated by the current residual in the Löwner
order.

As a result, much of the work on constructive approximation of positive
operators has followed one of two directions. One approach uses convex
optimization, often in the form of nuclear norm minimization \cite{MR2600248,MR2680543,MR2565240},
where positivity is enforced via spectral projections onto the positive
cone, implemented through repeated singular value or eigenvalue thresholding.
These methods are effective, but they typically require full spectral
computations at each step. A second line of work adapts greedy ideas
to matrices or trace class operators \cite{MR1806955,MR2821591,MR2848161}.
Here one must monitor positivity by step-size control or projections
onto $B\left(H\right)_{+}$, which complicates the iteration and weakens
the analogy with the vector case.

In this paper, we study a different type of decomposition scheme that
avoids additive updates entirely. We analyze the multiplicative residual
map 
\begin{equation}
R_{n+1}=\Phi_{u}\left(R_{n}\right):=R^{1/2}_{n}\left(I-\left|u\right\rangle \left\langle u\right|\right)R^{1/2}_{n},\label{eq:1-1}
\end{equation}
where $u\in H$ is a unit vector. The update \eqref{eq:1-1} serves
as a canonical, coordinate-free extraction mechanism inside the cone
$B\left(H\right)_{+}$. Unlike additive or spectral methods, it selects
the next rank-one direction through the intrinsic geometry of the
current residual. 

The expression is closely related to the Lüders update from quantum
measurement theory: for a projection $P$, a post-measurement state
is obtained by the compression $\rho\mapsto\left(I-P\right)\rho\left(I-P\right)$
(up to normalization; selective outcome, trace-nonincreasing) or by
the pinching $\rho\mapsto P\rho P+\left(I-P\right)\rho\left(I-P\right)$
(non-selective, trace-preserving), both linear completely positive
operations \cite{MR46289,MR725167,MR263379,MR2797301}. In contrast,
$\Phi_{u}\left(R\right)$ is nonlinear in $R$ and should be viewed
as a congruence/compression step. Here we treat it as an iterative
tool for constructing operator decompositions. This perspective makes
clear that the scheme defines a discrete-time dynamical system on
the cone $B(H)_{+}$.

The main results are as follows:

1. Canonical Decomposition and Frames. We prove that for any sequence
of directions, the residuals $R_{n}$ decrease monotonically to a
limit $R_{\infty}$. In the case of complete exhaustion ($R_{\infty}=0$),
we show that the extracted vectors $E_{n}=R^{1/2}_{n}u_{n}$ form
a Parseval frame for the range of $R_{0}$ in the sense of operator-valued
frame theory \cite{MR1686653,MR3329092,MR3495345}. This gives a constructive
method for producing Parseval frames directly from a positive operator,
without diagonalization.

2. Inverse Problem and Intrinsic Normalization. We characterize the
decreasing chains of positive operators with rank-one steps that arise
from the residual-weighted dynamics, via an intrinsic normalization
condition involving the Moore-Penrose inverses of the residuals.

3. Trace-class Energy Identity and Exhaustion Criteria. When the initial
operator is trace-class, we obtain a scalar energy identity for the
trace defect between the initial and limiting operators, and we prove
that weak greedy or density conditions on the chosen directions force
complete exhaustion and hence the Parseval frame property.

The paper is organized as follows. \prettyref{sec:2} introduces the
residual update and basic properties of $\Phi_{u}$. \prettyref{sec:3}
develops the dynamic decomposition, solves the inverse problem, and
describes the connection to Parseval frames. It also records the trace-class
energy identity, establishes the exhaustion criteria, and develops
an application to reproducing kernel Hilbert spaces, where the iteration
produces an explicit feature-map representation of a positive definite
kernel and yields an iterative kernel decomposition that avoids diagonalizing
the associated integral or Gram operator.

\section{Residual-weighted rank-one update}\label{sec:2}

Let $H$ be a complex Hilbert space. Let $B\left(H\right)_{+}$ denote
the set of bounded positive operators on $H$, partially ordered by
the Löwner order: 
\[
0\leq A\leq B\Longleftrightarrow0\leq\left\langle x,Ax\right\rangle \leq\left\langle x,Bx\right\rangle ,\quad\forall x\in H.
\]
Throughout, all inner products are linear in the second variable.
Denote by $\left|a\left\rangle \right\langle b\right|$ the rank-1
operator $x\mapsto\left\langle b,x\right\rangle a$. 
\begin{defn}
\label{def:b-1}Fix a unit vector $u\in H$ and the rank-one orthogonal
projection 
\[
P:=|u\left\rangle \right\langle u|.
\]
For all $R\in B\left(H\right)_{+}$, let 
\[
\Phi\left(R\right):=R^{1/2}\left(I-P\right)R^{1/2}.
\]
The operator $R\mapsto\Phi\left(R\right)$ will be referred to as
the residual-weighted rank-one update associated with\emph{ $u$}.
\end{defn}

\begin{rem}
By definition, $\Phi\left(R\right)=R-|R^{1/2}u\left\rangle \right\langle R^{1/2}u|$.
Thus, $ran\left(R-\Phi\left(R\right)\right)\subset ran\bigl(R^{1/2}\bigr)$
and $rank\left(R-\Phi\left(R\right)\right)\leq1$. In particular,
if $u\notin ker\left(R\right)$, then $\Phi\left(R\right)$ is obtained
from $R$ by removing the rank-one direction $R^{1/2}u$.
\end{rem}

\begin{lem}
\label{lem:b-3} $\Phi$ maps $B\left(H\right)_{+}$ into $B\left(H\right)_{+}$,
and $0\leq\Phi\left(R\right)\leq R$. 
\end{lem}

\begin{proof}
For all $x\in H$, 
\begin{align*}
\left\langle x,\Phi\left(R\right)x\right\rangle  & =\Vert\left(I-P\right)R^{1/2}x\Vert^{2}\geq0,
\end{align*}
so $\Phi:B\left(H\right)_{+}\hookrightarrow B\left(H\right)_{+}$.
Also, 
\[
R-\Phi\left(R\right)=R^{1/2}PR^{1/2}\geq0,
\]
i.e., $\Phi\left(R\right)\leq R$.
\end{proof}

\section{Canonical dynamic decomposition}\label{sec:3}

In this section, we show that the iterates of $\Phi$ form a monotone
sequence converging in the strong operator topology, identify the
limit, and express the initial operator as a telescoping sum of rank-one
terms. We then characterize complete exhaustion and connect the residual
directions to a Parseval frame on the natural range space, followed
by an inverse characterization that specifies exactly which rank-one
chains arise from this dynamics.

Fix an initial operator $R_{0}\in B\left(H\right)_{+}$ and a sequence
of unit vectors $\left(u_{n}\right)_{n\geq0}\subset H$. Consider
the residual-weighted rank-one updates
\begin{equation}
R_{n+1}=\Phi_{u_{n}}\left(R_{n}\right)=R^{1/2}_{n}\left(I-P_{n}\right)R^{1/2}_{n},\qquad P_{n}=\left|u_{n}\left\rangle \right\langle u_{n}\right|,\label{eq:3-1}
\end{equation}
whenever $R_{n}\neq0$. If $R_{n}=0$ for some $n$, then $R_{k}=0$
and $\Phi_{u_{k}}\left(R_{k}\right)=0$ for all $k\geq n$, so in
that case the iteration becomes stationary and there is nothing further
to prove. We therefore assume for the remainder of the section that
$R_{n}\neq0$ for all $n$.

It is convenient to introduce the residual vectors 
\[
E_{n}=R^{1/2}_{n}u_{n}\in H,\qquad n\geq0.
\]

The first step is to rewrite each update as the removal of a rank-one
positive operator:
\[
R_{n}-R_{n+1}=\left|E_{n}\left\rangle \right\langle E_{n}\right|,\qquad n\geq0.
\]

From this point on the sequence $\left(R_{n}\right)$ can be viewed
as a monotone decreasing sequence of positive operators, with each
difference $R_{n}-R_{n+1}$ explicitly identified as a rank-one term.
The next step is to make this precise and to identify the limiting
operator.
\begin{prop}
\label{prop:c-1} The sequence $\left(R_{n}\right)_{n\geq0}$ in \eqref{eq:3-1}
is decreasing in the Löwner order and bounded below by $0$. There
is a unique positive operator $R_{\infty}\in B\left(H\right)_{+}$
such that 
\[
\left\langle x,R_{\infty}x\right\rangle =\lim_{n\to\infty}\left\langle x,R_{n}x\right\rangle ,\qquad x\in H,
\]
and $R_{n}\to R_{\infty}$ in the strong operator topology.
\end{prop}

\begin{proof}
We have 
\[
R_{n}-R_{n+1}=\left|E_{n}\left\rangle \right\langle E_{n}\right|\geq0,\qquad n\geq0,
\]
so $R_{n+1}\leq R_{n}$ and the sequence $\left(R_{n}\right)$ is
monotone decreasing in the Löwner order. Since each $R_{n}$ is positive,
we also have $0\leq R_{n+1}\leq R_{n}$ for all $n$, so $R_{n}$
is bounded below by $0$.

Fix $x\in H$. Then $\left\langle x,R_{n+1}x\right\rangle \leq\left\langle x,R_{n}x\right\rangle $
for all $n$, and each term is nonnegative because $R_{n}\geq0$.
Hence the real sequence $\left\{ \left\langle x,R_{n}x\right\rangle \right\} _{n\geq0}$
is decreasing and bounded below by $0$, so the limit 
\[
\ell\left(x\right)=\lim_{n\to\infty}\left\langle x,R_{n}x\right\rangle 
\]
exists in $\left[0,\infty\right)$.

We now show that this limit arises from a unique positive operator
$R_{\infty}\in B\left(H\right)_{+}$. Define a quadratic form $q:H\to\left[0,\infty\right)$
by 
\[
q\left(x\right)=\ell\left(x\right)=\lim_{n\to\infty}\left\langle x,R_{n}x\right\rangle .
\]
Each functional $x\mapsto\left\langle x,R_{n}x\right\rangle $ is
continuous and $R_{n}\leq R_{0}$ implies 
\[
0\leq q\left(x\right)\leq\left\langle x,R_{0}x\right\rangle \leq\left\Vert R_{0}\right\Vert \left\Vert x\right\Vert ^{2}
\]
for all $x\in H$, so $q$ is dominated by a constant multiple of
$\left\Vert x\right\Vert ^{2}$. It follows that $q$ is a bounded
quadratic form on $H$. A standard argument (Lax-Milgram, or polarization
plus the Riesz representation theorem) now shows that there exists
a unique bounded selfadjoint operator $R_{\infty}\in B\left(H\right)$
such that 
\[
q\left(x\right)=\left\langle x,R_{\infty}x\right\rangle 
\]
for all $x\in H$. Since $q\left(x\right)\geq0$ for all $x$, the
operator $R_{\infty}$ is positive.

It remains to show that $R_{n}\to R_{\infty}$ strongly. Fix $x\in H$
and consider the sequence $R_{n}x$. For $m>n$ we have 
\[
R_{n}-R_{m}=\sum^{m-1}_{k=n}\left(R_{k}-R_{k+1}\right)=\sum^{m-1}_{k=n}\left|E_{k}\left\rangle \right\langle E_{k}\right|\geq0,
\]
so 
\[
0\leq\left\langle x,\left(R_{n}-R_{m}\right)x\right\rangle =\left\langle x,R_{n}x\right\rangle -\left\langle x,R_{m}x\right\rangle .
\]
Letting $m\to\infty$ and using the definition of $R_{\infty}$ we
obtain 
\[
\left\langle x,\left(R_{n}-R_{\infty}\right)x\right\rangle =\left\langle x,R_{n}x\right\rangle -\left\langle x,R_{\infty}x\right\rangle \to0\quad\text{as }n\to\infty.
\]
Since the operators are positive and bounded, convergence of the quadratic
forms $\left\langle x,R_{n}x\right\rangle $ to $\left\langle x,R_{\infty}x\right\rangle $
for every $x\in H$ implies $R_{n}x\to R_{\infty}x$ in norm for every
$x\in H$, that is, $R_{n}\to R_{\infty}$ in the strong operator
topology. The uniqueness of $R_{\infty}$ follows from uniqueness
of the quadratic form representation. For more details, see e.g.,
\cite{MR751959}.
\end{proof}
With the limiting operator in hand we can express $R_{0}$ as a sum
of the limit and an infinite series of rank-one positive operators
determined by the residual vectors $E_{n}$.
\begin{prop}
\label{prop:c-2} For every $N\geq1$ we have the finite telescoping
identity 
\[
R_{0}-R_{N}=\sum^{N-1}_{n=0}\left|E_{n}\left\rangle \right\langle E_{n}\right|.
\]
Moreover, for every $x\in H$, 
\[
\left\langle x,\left(R_{0}-R_{\infty}\right)x\right\rangle =\sum^{\infty}_{n=0}\left|\left\langle E_{n},x\right\rangle \right|^{2},
\]
and the series on the right is absolutely convergent. In particular,
\[
R_{0}-R_{\infty}=\sum^{\infty}_{n=0}\left|E_{n}\left\rangle \right\langle E_{n}\right|
\]
with convergence in the strong operator topology. 
\end{prop}

\begin{proof}
The finite telescoping identity is obtained by summing the differences:
\[
R_{0}-R_{N}=\sum^{N-1}_{n=0}\left(R_{n}-R_{n+1}\right)=\sum^{N-1}_{n=0}\left|E_{n}\left\rangle \right\langle E_{n}\right|,
\]
which is valid for all $N\geq1$.

Fix $x\in H$. Applying both sides of the finite telescoping identity
to $x$ and taking the inner product with $x$ gives 
\[
\left\langle x,\left(R_{0}-R_{N}\right)x\right\rangle =\sum^{N-1}_{n=0}\left\langle x,E_{n}\right\rangle \left\langle E_{n},x\right\rangle =\sum^{N-1}_{n=0}\left|\left\langle E_{n},x\right\rangle \right|^{2}.
\]
By \prettyref{prop:c-1} we know that $R_{N}\to R_{\infty}$ strongly,
so in particular 
\[
\left\langle x,\left(R_{0}-R_{N}\right)x\right\rangle \to\left\langle x,\left(R_{0}-R_{\infty}\right)x\right\rangle \quad\text{as }N\to\infty.
\]
Therefore the real sequence of partial sums 
\[
s_{N}\left(x\right)=\sum^{N-1}_{n=0}\left|\left\langle E_{n},x\right\rangle \right|^{2}
\]
is increasing (each term is nonnegative) and converges to the finite
limit 
\[
\lim_{N\to\infty}s_{N}\left(x\right)=\left\langle x,\left(R_{0}-R_{\infty}\right)x\right\rangle .
\]
This shows that $\sum^{\infty}_{n=0}\left|\left\langle E_{n},x\right\rangle \right|^{2}$
converges and that 
\[
\left\langle x,\left(R_{0}-R_{\infty}\right)x\right\rangle =\sum^{\infty}_{n=0}\left|\left\langle E_{n},x\right\rangle \right|^{2}.
\]

To see that the series of rank-one operators converges strongly, we
note that for each $N$, 
\[
T_{N}=\sum^{N-1}_{n=0}\left|E_{n}\left\rangle \right\langle E_{n}\right|=R_{0}-R_{N},
\]
hence $T_{N}\to R_{0}-R_{\infty}$ strongly because $R_{N}\to R_{\infty}$
strongly. Thus 
\[
R_{0}-R_{\infty}=\sum^{\infty}_{n=0}\left|E_{n}\left\rangle \right\langle E_{n}\right|
\]
with strong operator convergence. 
\end{proof}
The next result records a useful specialization in the trace-class
case, which makes the energy identity even more explicit.
\begin{prop}
\label{prop:c-3} Assume in addition that $R_{0}$ is trace-class.
Then each $R_{n}$ and $R_{\infty}$ is trace-class and 
\[
tr\left(R_{0}-R_{\infty}\right)=\sum^{\infty}_{n=0}\left\Vert E_{n}\right\Vert ^{2},
\]
with absolute convergence of the series on the right. In particular,
\[
\sum^{\infty}_{n=0}\left\Vert E_{n}\right\Vert ^{2}\leq\text{tr}\left(R_{0}\right)<\infty.
\]
\end{prop}

\begin{proof}
$R_{0}$ is trace-class and $R_{n}\leq R_{0}$ for all $n$, it follows
from standard properties of the trace that each $R_{n}$ is trace-class
and 
\[
0\leq tr\left(R_{n}\right)\leq tr\left(R_{0}\right)\quad\text{for all }n.
\]
Moreover, for each $n$, 
\[
R_{n}-R_{n+1}=\left|E_{n}\left\rangle \right\langle E_{n}\right|
\]
is rank-one, hence trace-class, and 
\[
tr\left(R_{n}-R_{n+1}\right)=tr\left(\left|E_{n}\left\rangle \right\langle E_{n}\right|\right)=\left\Vert E_{n}\right\Vert ^{2}.
\]
For each $N\geq1$ we now compute 
\[
tr\left(R_{0}-R_{N}\right)=tr\left(\sum^{N-1}_{n=0}\left(R_{n}-R_{n+1}\right)\right)=\sum^{N-1}_{n=0}tr\left(R_{n}-R_{n+1}\right)=\sum^{N-1}_{n=0}\left\Vert E_{n}\right\Vert ^{2},
\]
where we used linearity of the trace and the telescoping identity.

Since $R_{n}\downarrow R_{\infty}\geq0$ in the Löwner order, the
monotone convergence theorem for traces of positive operators implies
that $R_{\infty}$ is trace-class and 
\[
tr\left(R_{0}-R_{\infty}\right)=\lim_{N\to\infty}tr\left(R_{0}-R_{N}\right)=\lim_{N\to\infty}\sum^{N-1}_{n=0}\left\Vert E_{n}\right\Vert ^{2}=\sum^{\infty}_{n=0}\left\Vert E_{n}\right\Vert ^{2}.
\]
In particular, the series on the right converges and 
\[
\sum^{\infty}_{n=0}\left\Vert E_{n}\right\Vert ^{2}=tr\left(R_{0}-R_{\infty}\right)\leq tr\left(R_{0}\right).
\]
\end{proof}
Finally we formulate the main statement of this section, which collects
and summarizes the canonical decomposition and the uniqueness determined
by the update sequence.
\begin{thm}
\label{thm:c-4} Let $R_{0}\in B\left(H\right)_{+}$ and let $\left(u_{n}\right)_{n\geq0}\subset H$
be a sequence of unit vectors. Define $R_{n+1}=\Phi_{u_{n}}\left(R_{n}\right)$
and $E_{n}=R^{1/2}_{n}u_{n}$ for all $n\geq0$, and let $R_{\infty}$
be the strong limit of $\left(R_{n}\right)$ given by \prettyref{prop:c-1}.
Then:
\begin{enumerate}
\item $R_{n}\downarrow R_{\infty}$ in the Löwner order and $R_{n}\to R_{\infty}$
in the strong operator topology.
\item For every $x\in H$, 
\[
\left\langle x,\left(R_{0}-R_{\infty}\right)x\right\rangle =\sum^{\infty}_{n=0}\left|\left\langle E_{n},x\right\rangle \right|^{2},
\]
and the series converges.
\item The operator $R_{0}-R_{\infty}$ admits the strong operator decomposition
\[
R_{0}-R_{\infty}=\sum^{\infty}_{n=0}\left|E_{n}\left\rangle \right\langle E_{n}\right|.
\]
\item The vectors $E_{n}=R^{1/2}_{n}u_{n}$, and hence the rank-one operators
$\left|E_{n}\left\rangle \right\langle E_{n}\right|$, are uniquely
determined by $R_{0}$ and the update sequence $\left(u_{n}\right)$.
\end{enumerate}
\end{thm}

If in addition $R_{0}$ is trace-class, then 
\[
\text{\emph{tr}}\left(R_{0}-R_{\infty}\right)=\sum^{\infty}_{n=0}\left\Vert E_{n}\right\Vert ^{2}.
\]

\begin{proof}
Items (1)--(3) and the trace identity are exactly Propositions \ref{prop:c-1}--\ref{prop:c-3}.
The uniqueness of the vectors $E_{n}=R^{1/2}_{n}u_{n}$ follows from
the uniqueness of the positive square roots $R^{1/2}_{n}$ and the
recursive definition $R_{n+1}=\Phi_{u_{n}}\left(R_{n}\right)$, which
determines $R_{n}$ and hence $E_{n}$ inductively from $R_{0}$ and
$\left(u_{n}\right)$. 
\end{proof}

\subsection{Inverse Problem}

We consider the following inverse question for the residual-weighted
dynamics of this section: 
\begin{question*}
Which decreasing chains of positive operators with rank-one steps
arise from a residual-weighted iteration for some sequence of unit
vectors?
\end{question*}
The answer is expressed in terms of an intrinsic normalization condition
involving the Moore-Penrose inverse of the intermediate operators.
\begin{thm}
\label{thm:c-5} Let $R_{0}\in B\left(H\right)_{+}$ and let $\left(R_{n}\right)_{n\ge0}\subset B\left(H\right)_{+}$
and $\left(E_{n}\right)_{n\ge0}\subset H$ satisfy 
\[
R_{n+1}=R_{n}-\left|E_{n}\left\rangle \right\langle E_{n}\right|,\qquad n\ge0.
\]
Assume that:
\begin{enumerate}
\item For every $n$ with $E_{n}\neq0$, we have $E_{n}\in ran(R^{1/2}_{n})$; 
\item For every $n$, 
\begin{equation}
\left\langle E_{n},R^{\dagger}_{n}E_{n}\right\rangle =1,\label{eq:c-2}
\end{equation}
where $R^{\dagger}_{n}$ is the Moore-Penrose inverse of $R_{n}$. 
\end{enumerate}
Then there exists a sequence of unit vectors $\left(u_{n}\right)_{n\ge0}\subset H$
such that 
\[
R_{n+1}=\Phi_{u_{n}}\left(R_{n}\right)=R^{1/2}_{n}\left(I-P_{n}\right)R^{1/2}_{n},\qquad P_{n}=\left|u_{n}\left\rangle \right\langle u_{n}\right|,
\]
for all $n\ge0$, and 
\[
E_{n}=R^{1/2}_{n}u_{n},\qquad n\ge0.
\]
In particular, the chain $\left(R_{n}\right)$ arises from a residual
weighted iteration with residual vectors $E_{n}$.

Conversely, suppose $R_{0}\in B\left(H\right)_{+}$ and a sequence
of unit vectors $\left(u_{n}\right)_{n\ge0}\subset H$ are given such
that $u_{n}\in supp\left(R_{n}\right)$. Let 
\[
R_{n+1}=\Phi_{u_{n}}\left(R_{n}\right),\qquad E_{n}=R^{1/2}_{n}u_{n}.
\]
Then each $E_{n}$ lies in $ran(R^{1/2}_{n})$, and the normalization
condition \eqref{eq:c-2} holds for all $n\ge0$. 

\end{thm}

\begin{proof}
We first prove the converse direction. Fix $n$. Let $s_{n}=s\left(R_{n}\right)$
be the support projection of $R_{n}$, and note that $E_{n}\in ran(R^{1/2}_{n})\subset s_{n}H$.
On $s_{n}H$, the restriction $R_{n}|_{s_{n}H}$ is strictly positive
and invertible, with inverse $(R_{n}|_{s_{n}H})^{-1}$. The corresponding
Moore-Penrose inverse satisfies 
\[
R^{\dagger}_{n}x=(R_{n}|_{s_{n}H})^{-1}x,\qquad x\in s_{n}H.
\]
Since $E_{n}=R^{1/2}_{n}u_{n}$ lies in $s_{n}H$, we have
\[
R^{\dagger}_{n}E_{n}=(R_{n}|_{s_{n}H})^{-1}E_{n}=(R_{n}|_{s_{n}H})^{-1}R^{1/2}_{n}u_{n}=R^{-1/2}_{n}u_{n}
\]
(where all operators are understood as acting on $s_{n}H$). Therefore,
\[
\langle E_{n},R^{\dagger}_{n}E_{n}\rangle=\langle R^{1/2}_{n}u_{n},R^{-1/2}_{n}u_{n}\rangle=\langle u_{n},u_{n}\rangle=1.
\]

For the existence (forward) direction, assume that $R_{0}\in B\left(H\right)_{+}$
and sequences $\left(R_{n}\right),\left(E_{n}\right)$ satisfy the
hypotheses of the theorem. For each $n\ge0$, define the support projection
$s_{n}=s\left(R_{n}\right)$. Since $E_{n}\in ran(R^{1/2}_{n})$,
we have $E_{n}\in s_{n}H$. Define 
\[
u_{n}:=\left(R^{\dagger}_{n}\right)^{1/2}E_{n}.
\]
Then $u_{n}\in s_{n}H$ and 
\[
R^{1/2}_{n}u_{n}=R^{1/2}_{n}\left(R^{\dagger}_{n}\right)^{1/2}E_{n}=s_{n}E_{n}=E_{n}.
\]
Using the normalization condition, 
\[
\left\Vert u_{n}\right\Vert ^{2}=\left\langle u_{n},u_{n}\right\rangle =\langle\left(R^{\dagger}_{n}\right)^{1/2}E_{n},\left(R^{\dagger}_{n}\right)^{1/2}E_{n}\rangle=\left\langle E_{n},R^{\dagger}_{n}E_{n}\right\rangle =1.
\]
Thus $u_{n}$ is a unit vector in $H$ for every $n$. Finally, 
\[
\begin{aligned}\Phi_{u_{n}}\left(R_{n}\right) & =R_{n}-|R^{1/2}_{n}u_{n}\left\rangle \right\langle R^{1/2}_{n}u_{n}|\\
 & =R_{n}-\left|E_{n}\left\rangle \right\langle E_{n}\right|=R_{n+1}.
\end{aligned}
\]
This completes the proof.
\end{proof}

\subsection{Complete exhaustion and frames}\label{subsec:3-2}

The canonical dynamic decomposition in \prettyref{thm:c-4} expresses
the defect operator $\left(R_{0}-R_{\infty}\right)$ as a series of
rank-one terms built from the residual vectors $\left(E_{n}\right)$.
We now give a sharper characterization of the complete exhaustion
case $\left(R_{\infty}=0\right)$ in terms of the geometry of the
sequence $\left(E_{n}\right)$.

We begin with a simple observation.
\begin{prop}[Analysis and frame operators]
 \label{prop:c-6} Let $R_{0}\in B\left(H\right)_{+}$, let $\left(u_{n}\right)_{n\ge0}\subset H$
be a sequence of unit vectors, and let $\left(R_{n}\right)_{n\ge0}$,
$\left(E_{n}\right)_{n\ge0}$, and $R_{\infty}$ be as in \prettyref{thm:c-4}.
Define the analysis operator 
\[
T:H\to\ell^{2}\left(\mathbb{N}\right),\qquad Tx=\left(\left\langle E_{n},x\right\rangle \right)_{n\ge0}.
\]
Then $T$ is a bounded linear operator and the frame operator 
\[
S:=T^{*}T=\sum^{\infty}_{n=0}\left|E_{n}\left\rangle \right\langle E_{n}\right|
\]
exists in the strong operator topology with 
\[
S=R_{0}-R_{\infty}.
\]
\end{prop}

\begin{proof}
For each $x\in H$, \prettyref{thm:c-4} gives 
\[
\left\langle x,\left(R_{0}-R_{\infty}\right)x\right\rangle =\sum^{\infty}_{n=0}\left|\left\langle E_{n},x\right\rangle \right|^{2}.
\]
In particular, the series $\sum_{n\ge0}\left|\left\langle E_{n},x\right\rangle \right|^{2}$
converges and 
\[
\sum^{\infty}_{n=0}\left|\left\langle E_{n},x\right\rangle \right|^{2}=\left\langle x,\left(R_{0}-R_{\infty}\right)x\right\rangle \le\left\Vert R_{0}-R_{\infty}\right\Vert \left\Vert x\right\Vert ^{2}.
\]
Thus $T$ is well defined and bounded, with $\left\Vert T\right\Vert ^{2}\le\left\Vert R_{0}-R_{\infty}\right\Vert $.
By definition of $T$, 
\begin{align*}
\left\langle Tx,Ty\right\rangle _{\ell^{2}} & =\sum^{\infty}_{n=0}\left\langle x,E_{n}\right\rangle \left\langle E_{n},y\right\rangle =\left\langle x,\left(\sum\left|E_{n}\left\rangle \right\langle E_{n}\right|\right)y\right\rangle 
\end{align*}
whenever the series converges in the strong operator topology. On
the other hand, 
\[
\left\langle x,T^{*}Ty\right\rangle =\left\langle Tx,Ty\right\rangle _{\ell^{2}}
\]
for all $x,y\in H$. Comparing the two identities and using the characterization
\[
\left\langle x,\left(R_{0}-R_{\infty}\right)y\right\rangle =\sum^{\infty}_{n=0}\left\langle x,E_{n}\right\rangle \left\langle E_{n},y\right\rangle 
\]
by polarization, we obtain 
\[
T^{*}T=\sum^{\infty}_{n=0}\left|E_{n}\left\rangle \right\langle E_{n}\right|=R_{0}-R_{\infty}
\]
with strong operator convergence. 
\end{proof}
The next result is a sharper characterization of the complete exhaustion
case $\left(R_{\infty}=0\right)$. It shows that, in this regime,
the residual vectors form a Parseval-type family for the energy encoded
by $R_{0}$.
\begin{thm}
\label{thm:c-7} Let $R_{0}\in B\left(H\right)_{+}$, $\left(u_{n}\right)_{n\ge0}\subset H$
a sequence of unit vectors, and $\left(R_{n}\right)_{n\ge0}$, $\left(E_{n}\right)_{n\ge0}$,
$R_{\infty}$ as in \prettyref{thm:c-4}. Set 
\[
\mathcal{K}:=\overline{ran}(R^{1/2}_{0})\quad\text{and}\quad\mathcal{E}:=\overline{span}\left\{ E_{n}:n\ge0\right\} \subset H.
\]
Then the following are equivalent:
\begin{enumerate}
\item \label{enu:c-7-1}$R_{\infty}=0$.
\item \label{enu:c-7-2}For every $x\in H$, 
\[
\sum^{\infty}_{n=0}\left|\left\langle E_{n},x\right\rangle \right|^{2}=\left\langle x,R_{0}x\right\rangle .
\]
\item \label{enu:c-7-3}$\mathcal{E}=\mathcal{K}$ and, for every $x\in\mathcal{K}$,
\[
\sum^{\infty}_{n=0}\left|\left\langle E_{n},x\right\rangle \right|^{2}=\left\langle x,R_{0}x\right\rangle .
\]
\end{enumerate}
In particular, when $R_{\infty}=0$, the family $\left(E_{n}\right)_{n\ge0}$
is complete in $\mathcal{K}$, and the frame operator of $\left(E_{n}\right)$
coincides with $R_{0}$. 
\end{thm}

\begin{proof}
\eqref{enu:c-7-1}$\Rightarrow$\eqref{enu:c-7-2}. If $R_{\infty}=0$,
then \prettyref{thm:c-4} gives, for every $x\in H$, 
\[
\left\langle x,R_{0}x\right\rangle =\left\langle x,\left(R_{0}-R_{\infty}\right)x\right\rangle =\sum^{\infty}_{n=0}\left|\left\langle E_{n},x\right\rangle \right|^{2},
\]
which is \eqref{enu:c-7-2}.

\eqref{enu:c-7-2}$\Rightarrow$\eqref{enu:c-7-3}. First, $ran(R^{1/2}_{0})\subset\mathcal{E}$.
Indeed, by \prettyref{prop:c-6} we have $T^{*}T=R_{0}-R_{\infty}$.
Under assumption \eqref{enu:c-7-2} we obtain $R_{0}-R_{\infty}=R_{0}$,
so $R_{\infty}=0$ and therefore $T^{*}T=R_{0}$. The range of $T^{*}$
is the closed linear span of $\left\{ E_{n}\right\} $, that is, $\mathcal{E}$.
On the other hand, $ran\left(R_{0}\right)\subset\mathcal{E}$, because
\[
\mathrm{ran}\left(R_{0}\right)=ran\left(T^{*}T\right)\subset ran\left(T^{*}\right)\subset\mathcal{E}.
\]
Since $\mathcal{K}=\overline{ran}(R^{1/2}_{0})$ and $ran\left(R_{0}\right)$
is dense in $\overline{ran}(R^{1/2}_{0})$, we have 
\[
\mathcal{K}\subset\mathcal{E}.
\]
Conversely, each $E_{n}$ lies in $ran(R^{1/2}_{n})\subset ran(R^{1/2}_{0})$,
hence $\mathcal{E}\subset\mathcal{K}$. Therefore $\mathcal{E}=\mathcal{K}$,
and \eqref{enu:c-7-2} restricted to $x\in\mathcal{K}$ yields \eqref{enu:c-7-3}.

\eqref{enu:c-7-3}$\Rightarrow$\eqref{enu:c-7-1}. Let $x\in H$.
Decompose $x=x_{1}+x_{0}$ with $x_{1}\in\mathcal{K}$ and $x_{0}\in\mathcal{K}^{\perp}$.
Since $\mathcal{K}=\overline{ran}(R^{1/2}_{0})$, we have $R_{0}x_{0}=0$,
so 
\[
\left\langle x,R_{0}x\right\rangle =\left\langle x_{1},R_{0}x_{1}\right\rangle .
\]
On the other hand, each $E_{n}$ lies in $\mathcal{K}$, hence $\left\langle E_{n},x_{0}\right\rangle =0$,
and 
\[
\left\langle E_{n},x\right\rangle =\left\langle E_{n},x_{1}\right\rangle .
\]
Applying \eqref{enu:c-7-3} to $x_{1}\in\mathcal{K}$ gives 
\[
\sum^{\infty}_{n=0}\left|\left\langle E_{n},x\right\rangle \right|^{2}=\sum^{\infty}_{n=0}\left|\left\langle E_{n},x_{1}\right\rangle \right|^{2}=\left\langle x_{1},R_{0}x_{1}\right\rangle =\left\langle x,R_{0}x\right\rangle .
\]
Thus \eqref{enu:c-7-2} holds for all $x\in H$. Applying \prettyref{thm:c-4}
again, 
\[
\left\langle x,\left(R_{0}-R_{\infty}\right)x\right\rangle =\sum^{\infty}_{n=0}\left|\left\langle E_{n},x\right\rangle \right|^{2}=\left\langle x,R_{0}x\right\rangle 
\]
for all $x$, and therefore $\left\langle x,R_{\infty}x\right\rangle =0$
for all $x\in H$. Since $R_{\infty}\ge0$, this implies $R_{\infty}=0$. 
\end{proof}
\begin{prop}
Let $R_{0}$ be a trace-class positive operator. If the sequence of
unit vectors $(u_{n})_{n\ge0}$ is dense in the unit sphere of $H$,
then $R_{\infty}=0$.
\end{prop}

\begin{proof}
From \prettyref{prop:c-3}, we know that
\[
\sum^{\infty}_{n=0}\left\Vert E_{n}\right\Vert ^{2}=\sum^{\infty}_{n=0}\left\langle u_{n},R_{n}u_{n}\right\rangle =tr\left(R_{0}-R_{\infty}\right)<\infty.
\]
Since the sequence of operators is decreasing, $R_{n}\ge R_{\infty}\ge0$
for all $n$. This implies
\[
\left\langle u_{n},R_{n}u_{n}\right\rangle \ge\left\langle u_{n},R_{\infty}u_{n}\right\rangle \ge0.
\]
By the comparison test, the series $\sum^{\infty}_{n=0}\left\langle u_{n},R_{\infty}u_{n}\right\rangle $
converges.

Suppose for the sake of contradiction that $R_{\infty}\ne0$. Then
there exists a unit vector $v\in H$ and a constant $\delta>0$ such
that $\left\langle v,R_{\infty}v\right\rangle =\delta$. Since $x\mapsto\left\langle x,R_{\infty}x\right\rangle $
is continuous and the sequence $(u_{n})$ is dense in the unit sphere,
there exists a subsequence $(u_{n_{k}})$ and an open neighborhood
$U$ of $v$ such that $u_{n_{k}}\in U$ for all $k$, and $\left\langle u_{n_{k}},R_{\infty}u_{n_{k}}\right\rangle \ge\frac{\delta}{2}$
for all $k$. This implies that the sum $\sum^{\infty}_{n=0}\left\langle u_{n},R_{\infty}u_{n}\right\rangle $
has infinitely many terms bounded away from zero, and therefore must
diverge. This contradicts the convergence established above. Thus,
we must have $R_{\infty}=0$.
\end{proof}
\begin{cor}[Greedy Convergence]
\label{cor:c-9} Let $R_{0}$ be a trace-class positive operator.
Fix $c\in(0,1]$. If the unit vectors $(u_{n})_{n\ge0}$ satisfy
\[
\left\langle u_{n},R_{n}u_{n}\right\rangle \ge c\left\Vert R_{n}\right\Vert ,\quad\forall n
\]
then $R_{n}\to0$ in operator norm, hence $R_{\infty}=0$.
\end{cor}

\begin{proof}
By \prettyref{prop:c-3}, 
\[
\sum^{\infty}_{n=0}\left\langle u_{n},R_{n}u_{n}\right\rangle =tr(R_{0}-R_{\infty})<\infty,
\]
so $\langle u_{n},R_{n}u_{n}\rangle\to0$. The weak greedy condition
gives
\[
0\le c\left\Vert R_{n}\right\Vert \le\left\langle u_{n},R_{n}u_{n}\right\rangle \to0,
\]
hence $\left\Vert R_{n}\right\Vert \rightarrow0$. Therefore $R_{n}\to0$
in norm, which implies $R_{\infty}=0$.
\end{proof}
\begin{rem}[weaker hypothesis]
 Because $0\le R_{n+1}\le R_{n}$ implies $\left\Vert R_{n+1}\right\Vert \le\left\Vert R_{n}\right\Vert $,
it suffices that the weak greedy inequality hold for infinitely many
$n$. Along that subsequence the same limit argument gives $\left\Vert R_{n_{k}}\right\Vert \to0$,
and monotonicity then forces $\left\Vert R_{n}\right\Vert \to0$ for
the entire sequence.
\end{rem}

\begin{rem}[Constructive Parsevalization]
 The exhaustion results above admit a direct application to the construction
of frames. Assume $R_{\infty}=0$. If we initialize the dynamics with
$R_{0}=I$, then by \prettyref{thm:c-4}, the extracted vectors $E_{n}=R^{1/2}_{n}u_{n}$
satisfy 
\[
I=\sum^{\infty}_{n=0}\left|E_{n}\left\rangle \right\langle E_{n}\right|
\]
and thus form a Parseval frame for $H$. This identity holds in the
strong operator topology, and the reconstruction formula $x=\sum_{n\ge0}\left\langle E_{n},x\right\rangle E_{n}$
converges in norm for each $x\in H$. This process acts as a “soft”
alternative to Gram-Schmidt orthogonalization. In Gram-Schmidt, a
vector linearly dependent on its predecessors is mapped to zero (annihilated).
Here, the update $R_{n+1}=\Phi_{u_{n}}(R_{n})$ merely deflates the
energy in the redundant direction; the resulting vector $E_{n}$ is
nonzero but scaled down exactly enough to maintain the Parseval property.
Conceptually, this is a “covariance-action” analogue to the Kaczmarz
method: whereas Kaczmarz projects a vector to satisfy inner product
constraints, this iteration compresses the residual operator to satisfy
the operator-valued energy identity.
\end{rem}

\begin{example}[Iterative Kernel Feature Maps]
 Let $H_{K}$ be a Reproducing Kernel Hilbert Space (RKHS) on a domain
$\Omega$ (assume $\Omega$ is a separable metric space and $K$ is
continuous so that $x\mapsto k_{x}$ is norm-continuous) with reproducing
kernel $K(x,y)$. A central problem in learning theory is to find
a feature map $\psi:\Omega\to\ell^{2}$ such that 
\[
K\left(x,y\right)=\left\langle \psi\left(x\right),\psi\left(y\right)\right\rangle .
\]
While Mercer's theorem provides such a decomposition via eigenfunctions,
it requires diagonalizing the integral operator. The residual-weighted
decomposition offers an iterative construction. 

Let $\left\{ x_{n}\right\} _{n\ge0}\subset\Omega$ be a dense set
of sample points (hence $\left\{ k_{x_{n}}\right\} $ is total and
$\left\{ u_{n}\right\} $ is dense in the unit sphere).

Initialize $R_{0}=I$ (the identity on $H_{K}$). At step $n$, choose
the normalized kernel section $u_{n}=k_{x_{n}}/\left\Vert k_{x_{n}}\right\Vert $,
where $k_{x}=K\left(\cdot,x\right)$. Compute the residual feature
$E_{n}=R^{1/2}_{n}u_{n}$ and update $R_{n+1}=\Phi_{u_{n}}\left(R_{n}\right)$.

Assume the weak greedy condition $\left\langle u_{n},R_{n}u_{n}\right\rangle \ge c\left\Vert R_{n}\right\Vert $
for some $c\in(0,1]$ and all $n$. Then $R_{n}\to0$ in norm by \prettyref{cor:c-9},
hence $R_{\infty}=0$. Consequently, the resulting functions $\left\{ E_{n}\right\} $
form a Parseval frame for $H_{K}$. Applying the reproducing property
$\left\langle f,k_{x}\right\rangle =f(x)$, we recover the kernel
decomposition pointwise: 
\[
K\left(x,y\right)=\left\langle k_{x},k_{y}\right\rangle =\sum^{\infty}_{n=0}\left\langle k_{x},E_{n}\right\rangle \left\langle E_{n},k_{y}\right\rangle =\sum^{\infty}_{n=0}E_{n}\left(x\right)\overline{E_{n}\left(y\right)}.
\]
This approach generates the feature functions $E_{n}$ sequentially
and stably, avoiding the global inversion/diagonalization of a large
Gram or integral operator. 
\end{example}

\bibliographystyle{amsalpha}
\bibliography{ref}

\end{document}